\documentclass[preprint,12pt,a4paper]{elsarticle}

\usepackage{amsmath,amssymb,amsthm,mathtools,mathrsfs}
\usepackage{booktabs,array}
\usepackage{enumitem}
\usepackage{microtype}
\usepackage[T1]{fontenc}

\journal{Journal of Functional Analysis}
\biboptions{sort&compress}
\newtheorem{theorem}{Theorem}[section]
\newtheorem{proposition}[theorem]{Proposition}
\newtheorem{lemma}[theorem]{Lemma}
\newtheorem{corollary}[theorem]{Corollary}
\newtheorem{definition}[theorem]{Definition}
\newtheorem{remark}[theorem]{Remark}
\newtheorem{example}[theorem]{Example}

\newcommand{\Pp}{\mathbb P}
\newcommand{\R}{\mathbb R}
\newcommand{\Z}{\mathbb Z}
\newcommand{\K}{\mathcal K}
\newcommand{\Hs}{\mathscr H}
\newcommand{\Ff}{\mathcal F}

\newcommand{\cP}{\mathcal P}
\newcommand{\prox}{\operatorname{prox}}
\newcommand{\Fix}{\operatorname{Fix}}
\newcommand{\Ran}{\operatorname{Ran}}
\newcommand{\Var}{\operatorname{Var}}
\newcommand{\id}{\operatorname{id}}
\newcommand{\esssup}{\operatorname*{ess\,sup}}
\newcommand{\notinf}{\not\rightsquigarrow}

\begin{document}

\begin{frontmatter}

\title{Fibre-to-Section Lifting over Measure-Preserving Dynamics: Variational Reciprocity and Memory-Reversal Rigidity}

\author{Lei Luo\corref{cor1}}
\ead{cslluo@njust.edu.cn}

\cortext[cor1]{Corresponding author.}

\address{School of Computer Science and Engineering,
Nanjing University of Science and Technology,
Nanjing 210094, China}

\begin{abstract}
\sloppy
Many nonlinear and variational problems are posed fibrewise over a probability space, whereas the dynamically relevant objects are measurable sections coupled by measure-preserving transformations. We study structural obstructions created by this fibre-to-section lifting. First, for a global continuously differentiable potential on a real Hilbert space, absence of influence between closed subspaces is reciprocal; for finite orthogonal decompositions this yields an additive decomposition over dependency components. Second, finite Koopman memory channels are identifiable exactly under a sharp orbit-separation condition. Resolving the classical Hilbert-space potentiality criterion channel by channel then gives a shifted-adjoint balance between forward and reverse memories. At the global finite-difference level, if a section map with arbitrary finite dynamical memory is the gradient of a global continuously differentiable potential, its genuinely active lags are invariant under time reversal. Thus globally variational memory is reversal-complete, and one-sided memory collapses to present-state locality. Continuous proximity operators inherit the same rigidity through their convex-potential representation; a discontinuous proximal selection shows that continuity cannot generally be removed. Finally, under pointwise nondegeneracy of two coherent branches, vanishing branch-pasting residual is equivalent to asymptotic invariance of the branch labels. For finitely generated ergodic probability-preserving actions, strong ergodicity is therefore the exact threshold for relative compactness of all branch-pasting approximate-zero sequences. We also record a strong-compactness criterion for full measurable-selection lifts and a nonconvex Wasserstein minimizing-step illustration.
\end{abstract}


\end{frontmatter}

\section{Introduction}

\subsection{Fibre-to-section lifting}
Let $(\Omega,\Ff,\Pp)$ be a standard probability space and let $X$ be a metric or Hilbert space. A wide class of nonlinear problems is first defined separately at each $\omega\in\Omega$. A typical one-step relation is
\[
 y\in\Phi_\omega(x),
\]
or, in residual form,
\[
 G_\omega(x,y)=0.
\]
At the fibre level one may have compact solution sets, a variational principle, a proximal representation, or a well-behaved linearization. If a measure-preserving transformation $\theta$ couples the fibres, however, the dynamically relevant object is a measurable section $\Gamma:\Omega\to X$ satisfying
\begin{equation}\label{eq:intro-lift}
 G_\omega\bigl(\Gamma(\omega),\Gamma(\theta\omega)\bigr)=0.
\end{equation}
We call the passage from the fibre relation to the induced equation on a space of measurable sections \emph{fibre-to-section lifting}.

The guiding question of this paper is:
\begin{quote}
Which fibrewise analytic structures survive after the fibres are coupled into a stationary section space?
\end{quote}
The answer is controlled by a common mechanism that is easy to miss when compactness, potentiality and proximal representability are considered separately: global variational structure imposes reciprocity on influence, while dynamical lifting naturally creates oriented dependencies.

\subsection{Variational reciprocity as the organizing principle}
Let $\K$ be a real Hilbert space and $T:\K\to\K$. For closed subspaces $M,N\subset\K$, write
\[
 M\notinf_TN
\]
when perturbations in $M$ never alter the $N$-component of the output. If $T=\nabla\Psi$ for a global $C^1$ potential, we prove
\[
 M\notinf_TN\quad\Longleftrightarrow\quad N\notinf_TM.
\]
Neither convexity of $\Psi$ nor orthogonality of $M$ and $N$ is needed. This is a finite-difference consequence of global exactness, and we use it as an organizing principle rather than as a new abstract replacement for classical gradient symmetry.

For a finite orthogonal block decomposition, reciprocity makes the exact block-dependency graph undirected. Its connected components $K_1,\dots,K_r$ determine an additive splitting
\[
 \Psi(x)=C+\sum_{\alpha=1}^r\Psi_\alpha(P_{K_\alpha}x).
\]
For convex proximal maps, the same components split the generating penalty. Interaction graphs and partial separability have a long history in structured optimization \cite{GriewankToint1984}; the role of our block theorem is to derive the decomposition directly from global gradient dependence and then use it as input to the dynamical memory results below.

\subsection{Dynamical memories and two manifestations of reciprocity}
Let $H$ be a real separable Hilbert space and
\[
 \Hs=L^2(\Omega;H).
\]
For an invertible p.m.p.\ transformation $\theta$ define the Koopman memories
\[
 (U^ku)(\omega)=u(\theta^k\omega),\qquad k\in\Z.
\]
For a finite $F\subset\Z$, introduce
\begin{equation}\label{eq:sep-intro}
 (\mathrm{Sep}_F):\qquad
 \Pp\{\omega:\theta^k\omega=\theta^\ell\omega\}=0
 \quad(k\ne\ell,\ k,\ell\in F).
\end{equation}
We show that this is exactly the condition under which arbitrary bounded operator-valued channels are identifiable:
\[
 \sum_{k\in F}M_{A_k}U^k=0
 \quad\Longrightarrow\quad A_k=0\ \text{a.e.\ for every }k\in F.
\]
The converse is also true. If $\theta$ is aperiodic, $(\mathrm{Sep}_F)$ is automatic for every finite $F$. This concrete Koopman--multiplication representation should be distinguished from the regular Fourier representation in group-measure-space crossed products: coefficient uniqueness can fail here when orbit channels collide on a set of positive measure \cite{Takesaki2003}.

Channel identifiability yields the first, differential manifestation of memory reversal. For a regular finite-memory field
\[
 DF(u)=\sum_{k\in S}M_{A_k[u]}U^k,
\]
potentiality is equivalent to
\[
 A_k[u](\omega)=A_{-k}[u](\theta^k\omega)^*.
\]
Thus smooth variationality pairs every active derivative channel with its time reverse through a shifted adjoint.

The main global result is the finite-difference counterpart. Suppose
\[
 (Tu)(\omega)=t_\omega\bigl((u(\theta^s\omega))_{s\in S}\bigr)
\]
for a finite, not necessarily one-sided, memory set $S\subset\Z$. Under a finite orbit-separation condition, if $T=\nabla\Psi$ for a global $C^1$ potential on $\Hs$, then the active memory set satisfies
\[
 \boxed{\operatorname{Act}(T)=-\operatorname{Act}(T).}
\]
Hence global variational representability forces reversal-complete memory. The previous one-sided no-go theorem is the extremal case: if $S\cap(-S)\subset\{0\}$, all nonzero lags are inactive. Continuous proximity operators inherit this conclusion from the Gribonval--Nikolova convex-potential representation \cite{GribonvalNikolova2020}.

\subsection{Compactness and strong-ergodicity thresholds}
Two complementary lifting obstructions concern compactness. First, a full decomposable selection family over a nonatomic base is strongly compact only when the fibre correspondence is essentially single-valued. Thus fibrewise compactness does not survive arbitrary measurable pasting.

Second, suppose a residual admits two coherent branches $\xi_0,\xi_1$. Pasting the branches according to a label set $A$ produces an error only across the dynamical boundary $A\triangle\theta^{-1}A$. With pointwise nondegenerate cross-branch residuals, branch-pasting residual tends to zero if and only if the labels become asymptotically invariant. For a finitely generated ergodic p.m.p.\ group action this gives an exact threshold:
\[
 \begin{gathered}
 \text{strong ergodicity}\\
 \Longleftrightarrow\\
 \text{relative compactness of all branch-pasting approximate-zero sequences}.
 \end{gathered}
\]
Strong ergodicity itself is classical \cite{Schmidt1981}; our contribution is the nonlinear lifting consequence. For an aperiodic ergodic $\Z$-action, Rokhlin towers yield nontrivial asymptotically invariant labels and therefore noncompact approximate solutions.

\subsection{Contribution hierarchy}
The contributions are deliberately hierarchical rather than a list of parallel obstructions.
\begin{enumerate}[label=(C\arabic*)]
\item \textbf{Variational reciprocity as structural input.} We isolate a $C^1$ Hilbert-space finite-difference reciprocity principle and the associated finite block-component decomposition. This supplies the global support calculus used later; it is not presented as a replacement for classical exactness theory.
\item \textbf{Channel identifiability and smooth memory reversal.} We identify the exact finite orbit-separation condition for uniqueness of Koopman--multiplication channels and resolve the classical Helmholtz condition into a shifted-adjoint pairing of forward and reverse memories, including periodic aliasing.
\item \textbf{Global finite-memory reversal rigidity.} For arbitrary finite dynamical memory, we prove that the active lag set of a globally $C^1$-variational section map is invariant under $k\mapsto-k$. One-sided global variational rigidity and the corresponding continuous-proximal no-go statements follow as corollaries. This is the principal global rigidity theorem of the paper.
\item \textbf{Branch-pasting compactness threshold.} We identify vanishing branch-pasting residual exactly with asymptotic label invariance. Strong ergodicity then becomes the exact compactness threshold for this two-branch approximate-zero class.
\item \textbf{Complementary selection compactness.} We retain a sharp strong-compactness obstruction for full measurable-selection lifts and an explicit nonconvex Wasserstein illustration.
\end{enumerate}

\subsection{Relation to classical theories}
The abstract criterion $F=\nabla E\iff DF=DF^*$ on a convex Hilbert domain is classical potential-operator theory; see Vainberg \cite{Vainberg1964}. Forward/backward identities also occur in inverse variational problems with delayed or deviated arguments \cite{Popov1998,KolesnikovaPopovSavchin2007}, as well as in discrete Helmholtz conditions \cite{BourdinCresson2013}. The atomic/coatomic framework of Drakhlin, Ponosov and Stepanov \cite{DrakhlinPonosovStepanov2002} provides a measure-theoretic language for shifts and local memories. Our smooth theorem does not claim memory reversal itself as a new variational principle. It identifies the exact p.m.p.\ channel-separation mechanism and its operator-valued shifted-adjoint resolution.

Partially separable optimization has long used sparse interaction structures and additive element decompositions \cite{GriewankToint1984}. We therefore do not claim that an undirected interaction graph is new terminology. Our dependency theorem starts instead from input-output independence of a global Hilbert gradient field, obtains component decomposition without assuming a pre-existing partial-separability representation, and serves primarily as the finite-difference engine for the dynamical memory-reversal theorem.

Moreau's theory \cite{Moreau1965} and the characterization of Gribonval and Nikolova \cite{GribonvalNikolova2020} provide the proximal input. We do not propose a new characterization of proximity operators. Triangular cyclically monotone maps have a close finite-dimensional relative in recent transport work of De Lara and Ganassali \cite{DeLaraGanassali2025}; our one-sided proximal statements are presented as consequences of the more general global variational memory theorem.

The channel-identifiability calculation resembles Fourier-coefficient uniqueness in crossed products. Our operators, however, act in the concrete Koopman--multiplication representation on $L^2(\Omega;H)$, where orbit collisions cause genuine aliasing; see Takesaki \cite{Takesaki2003} for the operator-algebraic background. Finally, strong ergodicity and asymptotically invariant sets are classical \cite{Schmidt1981}. The new statement on that side is the nonlinear identification of asymptotic label invariance with branch-pasting residual. The strong-ergodicity compactness threshold follows from that identification.

\subsection{Organization}
Section 2 sets up section spaces, measurable multifunctions and dynamical memories. Section 3 records variational reciprocity and proves the finite block-component decomposition. Section 4 establishes channel identifiability and adjoint calculus. Section 5 gives the measure-dynamical Helmholtz theorem. Section 6 proves global finite-memory reversal rigidity and derives the one-sided and proximal no-go consequences, including a sharp discontinuous counterexample. Section 7 treats strong compactness of full selection lifts. Section 8 proves the branch-pasting/asymptotic-invariance theorem and the strong-ergodicity threshold. Section 9 gives variational and Wasserstein examples. Section 10 records scope and limitations. Technical material is collected in the appendices.

\section{Section lifts over measure-preserving dynamics}

Throughout, $(\Omega,\Ff,\Pp)$ is a standard probability space and $1\le p<\infty$ unless stated otherwise.

\subsection{Metric-valued Lp sections}
Let $(X,d)$ be Polish and fix $x_\circ\in X$. We write $L^p(\Omega;X)$ for equivalence classes of strongly measurable $u:\Omega\to X$ such that
\[
 \int_\Omega d(u(\omega),x_\circ)^p\,d\Pp(\omega)<\infty,
\]
with metric
\begin{equation}\label{eq:Dp}
 D_p(u,v)=\left(\int_\Omega d(u(\omega),v(\omega))^p\,d\Pp(\omega)\right)^{1/p}.
\end{equation}
When $X=H$ is a separable Hilbert space we use Bochner notation and set $\Hs=L^2(\Omega;H)$.

\subsection{Atomic and nonatomic components}
Modulo null sets, a standard probability space decomposes as
\begin{equation}\label{eq:atomic-decomp}
 \Omega=\Omega_{\rm na}\sqcup\bigsqcup_{j\ge1}A_j,
\end{equation}
where $\Omega_{\rm na}$ is nonatomic and the pairwise disjoint $A_j$ are atoms of positive measure.

\begin{lemma}\label{lem:atom-constant}
Let $A$ be an atom of positive measure and $X$ separable metric. Every measurable $u:A\to X$ is almost everywhere constant.
\end{lemma}
\begin{proof}
Normalize $\Pp|_A$ and push it forward by $u$. If the image probability were not a Dirac mass, separability would provide a Borel set of image measure strictly between zero and one. Its inverse image would contradict atomicity.
\end{proof}

\subsection{Measurable multifunctions and full selection lifts}
Let $C:\Omega\rightrightarrows X$ be measurable, nonempty closed-valued. By the Castaing representation theorem \cite{CastaingValadier1977}, there are measurable selections $c_n$ such that
\begin{equation}\label{eq:castaing}
 C(\omega)=\overline{\{c_n(\omega):n\ge1\}}
 \quad\text{for a.e.\ }\omega.
\end{equation}
We call $C$ $p$-integrably bounded if
\[
 r_C(\omega):=\sup_{x\in C(\omega)}d(x,x_\circ)\in L^p(\Omega).
\]
Define
\begin{equation}\label{eq:selector-lift}
 S_p(C)=\{u\in L^p(\Omega;X):u(\omega)\in C(\omega)\ \text{a.e.}\}.
\end{equation}
A set $D\subset L^p(\Omega;X)$ is \emph{decomposable} if measurable pasting of any two of its elements along any measurable set remains in $D$. Every full selection lift $S_p(C)$ is decomposable.

\subsection{Dynamical memories}
Let $\theta:\Omega\to\Omega$ be invertible, bimeasurable and measure preserving. For $k\in\Z$ define
\begin{equation}\label{eq:Uk}
 (U^ku)(\omega)=u(\theta^k\omega).
\end{equation}
Each $U^k$ is an isometry on $L^p(\Omega;X)$; on $\Hs=L^2(\Omega;H)$,
\[
 (U^k)^*=U^{-k}.
\]
A finite-memory field has the schematic form
\begin{equation}\label{eq:finite-memory}
 (Fu)(\omega)=f_\omega\bigl((u(\theta^k\omega))_{k\in S}\bigr),
 \qquad S\subset\Z\ \text{finite}.
\end{equation}

\section{Variational reciprocity and dependency graphs}

This section contains the general structural principle used later in the dynamical no-go theorem. It does not rely on convexity or on proximal structure.

\subsection{Dependency between closed subspaces}
Let $\K$ be a real Hilbert space and let $M,N\subset\K$ be closed linear subspaces with orthogonal projections $P_M,P_N$.

\begin{definition}[Absence of influence]\label{def:influence}
For a map $T:\K\to\K$, we write
\[
 M\notinf_TN
\]
if
\begin{equation}\label{eq:no-influence}
 P_NT(x+h)=P_NT(x)
 \qquad\forall x\in\K,\ \forall h\in M.
\end{equation}
\end{definition}
No relation between $M$ and $N$ is assumed in this definition.

\begin{proposition}[Variational reciprocity]\label{prop:variational-reciprocity}
Let $\Psi\in C^1(\K)$ and $T=\nabla\Psi$. For any closed linear subspaces $M,N\subset\K$,
\begin{equation}\label{eq:reciprocity}
 M\notinf_TN\quad\Longleftrightarrow\quad N\notinf_TM.
\end{equation}
\end{proposition}

\begin{proof}
Assume $M\notinf_TN$. Fix $x\in\K$, $h\in M$ and $k\in N$, and define the mixed increment
\[
 \Delta_{h,k}\Psi(x)
 =\Psi(x+h+k)-\Psi(x+h)-\Psi(x+k)+\Psi(x).
\]
Integrating first in the $k$-direction gives
\[
 \Delta_{h,k}\Psi(x)
 =\int_0^1\langle T(x+h+tk)-T(x+tk),k\rangle\,dt.
\]
Since $k\in N$, the inner product only sees the $N$-projection of the difference. By $M\notinf_TN$ this projection vanishes, hence $\Delta_{h,k}\Psi(x)=0$.

Replace $h$ by $sh$, $s\in\R$. The same identity holds for every $s$. Differentiating at $s=0$ yields
\[
 \langle T(x+k)-T(x),h\rangle=0
 \qquad\forall h\in M.
\]
Thus $P_MT(x+k)=P_MT(x)$ for every $k\in N$, i.e.\ $N\notinf_TM$. Interchanging $M$ and $N$ gives the converse.
\end{proof}

\begin{remark}[What is and is not new]\label{rem:exactness}
For $C^2$ potentials, infinitesimal reciprocity is encoded by symmetry of the Hessian. Proposition~\ref{prop:variational-reciprocity} is a global finite-difference formulation requiring only $C^1$ regularity of the potential. We use it as an organizing structural principle rather than as a replacement for classical exactness theory.
\end{remark}

\subsection{Finite block-dependency graphs}
Let
\begin{equation}\label{eq:block-decomp}
 \K=H_1\oplus\cdots\oplus H_m
\end{equation}
be a finite orthogonal decomposition. For $i\ne j$, declare that $i\to j$ if $H_i\notinf_TH_j$ fails. This is the \emph{exact dependency digraph} of $T$ with respect to \eqref{eq:block-decomp}.

\begin{theorem}[Dependency-component decomposition]\label{thm:graph-classification}
Let $T=\nabla\Psi$ with $\Psi\in C^1(\K)$.
\begin{enumerate}[label=(\roman*)]
\item The exact dependency digraph is symmetric: $i\to j$ if and only if $j\to i$. Hence it may be identified with an undirected graph.
\item Let $C_1,\dots,C_r$ be its connected components and
\[
 K_\alpha=\bigoplus_{i\in C_\alpha}H_i.
\]
Then there are $C^1$ functions $\Psi_\alpha:K_\alpha\to\R$ and $C\in\R$ such that
\begin{equation}\label{eq:potential-component-decomp}
 \Psi(x)=C+\sum_{\alpha=1}^r\Psi_\alpha(P_{K_\alpha}x),
\end{equation}
and
\begin{equation}\label{eq:T-component-decomp}
 T(x)=\bigoplus_{\alpha=1}^r\nabla\Psi_\alpha(P_{K_\alpha}x).
\end{equation}
\end{enumerate}
\end{theorem}

\begin{proof}
Part (i) is Proposition~\ref{prop:variational-reciprocity} applied to $H_i,H_j$.

For (ii), fix a component $C_\alpha$. Since there is no edge from any $H_j$ outside $C_\alpha$ to any $H_i$ inside $C_\alpha$, changing one outside block leaves $P_{K_\alpha}T$ unchanged. Changing the finitely many outside blocks successively shows that $P_{K_\alpha}T(x)$ depends only on $P_{K_\alpha}x$. Hence there is a map $T_\alpha:K_\alpha\to K_\alpha$ such that
\[
 P_{K_\alpha}T(x)=T_\alpha(P_{K_\alpha}x).
\]
Restrict $\Psi$ to $K_\alpha$ and define
\[
 \Psi_\alpha(z)=\Psi(z)-\Psi(0),\qquad z\in K_\alpha,
\]
where $z$ is embedded in $\K$ with all other component coordinates zero. Then
\[
 \nabla_{K_\alpha}\Psi_\alpha(z)=T_\alpha(z).
\]
The function
\[
 \widetilde\Psi(x)=\Psi(0)+\sum_{\alpha=1}^r\Psi_\alpha(P_{K_\alpha}x)
\]
has gradient $T=\nabla\Psi$. Since $\K$ is connected, $\Psi-\widetilde\Psi$ is constant; both agree at zero, so they are equal. This proves \eqref{eq:potential-component-decomp} and \eqref{eq:T-component-decomp}.

\end{proof}

\begin{remark}[Acyclic prescribed architectures]
As an immediate graph-theoretic consequence of reciprocity, if an exact block-dependency digraph of a global gradient field is directed and acyclic, then it has no cross-block edges: any edge $i\to j$ would force the two-cycle $i\leftrightarrow j$.
\end{remark}

\begin{corollary}[Triangular gradient fields collapse]\label{cor:triangular-gradient}
Suppose $T=\nabla\Psi$ is causal with respect to the order in \eqref{eq:block-decomp}, in the sense that the first $j$ output blocks depend only on the first $j$ input blocks for every $j$. Then $T$ is block diagonal.
\end{corollary}
\begin{proof}
Causality rules out dependence of an earlier output block on any later input block. Proposition~\ref{prop:variational-reciprocity} rules out the reciprocal dependence as well. Applying this to each pair of distinct blocks leaves only the diagonal dependencies.
\end{proof}

\subsection{Convex proximal component decomposition}
Recall $\Gamma_0(\K)$ denotes the proper lower-semicontinuous convex functions $\Phi:\K\to\R\cup\{+\infty\}$.

\begin{theorem}[Dependency components of convex proximal maps]\label{thm:prox-components}
Let $\Phi\in\Gamma_0(\K)$ and $T=\prox_\Phi$. Let $C_1,\dots,C_r$ be the connected components of the exact dependency graph of $T$, with associated orthogonal sums $K_\alpha$. Then there exist $\phi_\alpha\in\Gamma_0(K_\alpha)$ and $C\in\R$ such that
\begin{equation}\label{eq:penalty-component-decomp}
 \Phi(x)=C+\sum_{\alpha=1}^r\phi_\alpha(P_{K_\alpha}x).
\end{equation}
Conversely, any decomposition of the form \eqref{eq:penalty-component-decomp} makes $\prox_\Phi$ block diagonal across the $K_\alpha$.
\end{theorem}

\begin{proof}
By Moreau's theory, $T$ is firmly nonexpansive and is the gradient of a convex $C^1$ potential. Theorem~\ref{thm:graph-classification} gives
\[
 T=\bigoplus_{\alpha=1}^rT_\alpha
\]
with $T_\alpha:K_\alpha\to K_\alpha$. Firm nonexpansiveness passes to each block, and each $T_\alpha$ is itself the gradient of a convex $C^1$ potential. Moreau's characterization therefore yields $\phi_\alpha\in\Gamma_0(K_\alpha)$ with $T_\alpha=\prox_{\phi_\alpha}$.

Set $\widetilde\Phi(x)=\sum_\alpha\phi_\alpha(P_{K_\alpha}x)$. Orthogonal separation of the minimization gives
\[
 \prox_{\widetilde\Phi}=\bigoplus_\alpha\prox_{\phi_\alpha}=T=\prox_\Phi.
\]
Let $e_\Phi,e_{\widetilde\Phi}$ be the unit-parameter Moreau envelopes. Since
\[
 \nabla e_\Phi=I-\prox_\Phi=I-\prox_{\widetilde\Phi}=\nabla e_{\widetilde\Phi},
\]
the envelopes differ by a constant. Using
\[
 e_\Phi^*=\Phi^*+\tfrac12\|\cdot\|^2
\]
and Fenchel conjugation gives $\Phi=\widetilde\Phi+C$ after renaming the constant. The converse follows directly by separability of the proximal minimization.
\end{proof}

\begin{corollary}[Triangular convex proximal maps]\label{cor:triangular-prox}
For $\Phi\in\Gamma_0(H_1\oplus\cdots\oplus H_m)$, if $\prox_\Phi$ is triangular with respect to the ordered blocks, then
\[
 \Phi(x_1,\dots,x_m)=C+\sum_{i=1}^m\phi_i(x_i)
\]
for some $\phi_i\in\Gamma_0(H_i)$. The converse is immediate.
\end{corollary}

\section{Identifiability and adjoint calculus for dynamical memories}

We return to $\Hs=L^2(\Omega;H)$ with $H$ real separable and $\theta$ invertible p.m.p.

\subsection{Strong-operator measurable coefficients}
Since $\mathcal L(H)$ need not be norm-separable, let $\mathcal M_s^\infty(\Omega;\mathcal L(H))$ denote the fields $A$ for which $\omega\mapsto A(\omega)v$ is strongly measurable for every $v\in H$ and
\[
 \esssup_\omega\|A(\omega)\|<\infty.
\]
Such $A$ defines a bounded multiplication operator $M_A$ on $\Hs$.

\begin{lemma}\label{lem:adjoint-measurable}
If $A\in\mathcal M_s^\infty(\Omega;\mathcal L(H))$, then $A^*$ belongs to the same class.
\end{lemma}
\begin{proof}
For fixed $v,w\in H$, $\langle A(\omega)^*v,w\rangle=\langle v,A(\omega)w\rangle$ is measurable. Thus $A(\cdot)^*v$ is weakly measurable. Separability and Pettis' theorem give strong measurability; the essential norm bound is unchanged.
\end{proof}

\subsection{Finite-channel separation}
For finite $F\subset\Z$ define $(\mathrm{Sep}_F)$ by \eqref{eq:sep-intro}.

\begin{theorem}[Identifiability of dynamical memory channels]\label{thm:identifiability}
Let $F\subset\Z$ be finite. The following are equivalent:
\begin{enumerate}[label=(\roman*)]
\item $\theta$ satisfies $(\mathrm{Sep}_F)$;
\item for every family $A_k\in\mathcal M_s^\infty(\Omega;\mathcal L(H))$,
\begin{equation}\label{eq:channel-sum}
 \sum_{k\in F}M_{A_k}U^k=0
\end{equation}
implies $A_k=0$ a.e.\ for every $k\in F$.
\end{enumerate}
\end{theorem}

\begin{proof}
Assume $(\mathrm{Sep}_F)$. On a full-measure standard Borel subset the finite orbit points $\theta^k\omega$, $k\in F$, are pairwise distinct. A countable basis yields a countable measurable cover by sets $E$ such that $\theta^kE$, $k\in F$, are pairwise disjoint. Fix $j\in F$, one such $E$ and $v\in H$. With $h=1_{\theta^jE}v$, for $\omega\in E$ all channels in \eqref{eq:channel-sum} vanish except $j$, hence the value is $A_j(\omega)v$. A countable dense subset of $H$ and bounded linearity give $A_j=0$ on $E$. The countable cover and arbitrariness of $j$ prove uniqueness.

Conversely, if $(\mathrm{Sep}_F)$ fails, there are $k\ne\ell$ and a positive-measure set $B$ on which $\theta^k\omega=\theta^\ell\omega$. Set $A_k=1_BI_H$, $A_\ell=-1_BI_H$ and all other coefficients zero. The operator sum vanishes although the coefficients do not.
\end{proof}

\begin{remark}[Aperiodic bases]\label{rem:aperiodic-sep}
If $\theta$ is aperiodic, then $\Pp(\Fix(\theta^q))=0$ for every $q\ne0$. Hence $(\mathrm{Sep}_F)$ holds automatically for every finite $F\subset\Z$. No ergodicity assumption is needed for this observation.
\end{remark}

\begin{remark}[Periodic aliasing]\label{rem:aliasing}
If $\theta^q=\id$ on an invariant component, then $U^{k+q}=U^k$ there. Individual exponents are identifiable only modulo $q$; the natural coefficients are the residue-class aggregates
\[
 B_r=\sum_{k\equiv r\,({\rm mod}\ q)}A_k.
\]
\end{remark}

\subsection{Adjoint calculus}
\begin{lemma}[Adjoint of one channel]\label{lem:channel-adjoint}
For $A\in\mathcal M_s^\infty(\Omega;\mathcal L(H))$ and $k\in\Z$,
\begin{equation}\label{eq:channel-adjoint}
 (M_AU^k)^*=M_{A^*\circ\theta^{-k}}U^{-k}.
\end{equation}
\end{lemma}
\begin{proof}
For $u,v\in\Hs$, measure preservation and $\eta=\theta^k\omega$ give
\[
 \begin{aligned}
 \langle M_AU^ku,v\rangle
 &=\int\langle A(\omega)u(\theta^k\omega),v(\omega)\rangle\,d\Pp(\omega)\\
 &=\int\langle u(\eta),A(\theta^{-k}\eta)^*v(\theta^{-k}\eta)\rangle\,d\Pp(\eta).
 \end{aligned}
\]
which is \eqref{eq:channel-adjoint}.
\end{proof}

\begin{corollary}\label{cor:sum-adjoint}
If $T=\sum_{k\in F}M_{A_k}U^k$, then
\[
 T^*=\sum_{j\in-F}M_{A_{-j}^*\circ\theta^j}U^j.
\]
\end{corollary}

\subsection{Regular finite-memory fields}
Let $S\subset\Z$ be finite and $O\subset\Hs$ open.

\begin{definition}\label{def:regular-memory}
A $C^1$ map $F:O\to\Hs$ is a \emph{regular $S$-memory field} if
\begin{equation}\label{eq:regular-derivative}
 DF(u)=\sum_{k\in S}M_{A_k[u]}U^k,
\end{equation}
where $A_k[u]\in\mathcal M_s^\infty(\Omega;\mathcal L(H))$. We set $A_k[u]=0$ for $k\notin S$.
\end{definition}

\section{Measure-dynamical Helmholtz rigidity}

\subsection{Classical potentiality criterion}
\begin{proposition}[Potentiality criterion]\label{prop:Vainberg}
Let $O$ be convex and open in a real Hilbert space and let $F\in C^1(O;\Hs)$. The following are equivalent:
\begin{enumerate}[label=(\roman*)]
\item there exists $E\in C^2(O;\R)$ with $F=\nabla E$;
\item $DF(u)$ is self-adjoint for every $u\in O$.
\end{enumerate}
If they hold, then for fixed $u_0\in O$,
\begin{equation}\label{eq:radial-potential}
 E(u)=C+\int_0^1\langle F(u_0+t(u-u_0)),u-u_0\rangle\,dt
\end{equation}
is a potential.
\end{proposition}
\begin{proof}
Necessity is symmetry of the Hessian. Conversely, differentiation of \eqref{eq:radial-potential}, followed by self-adjointness of $DF$, yields $DE(u)h=\langle F(u),h\rangle$. This is classical potential-operator theory \cite{Vainberg1964}.
\end{proof}

\subsection{Channelwise Helmholtz theorem}
\begin{theorem}[Measure-dynamical Helmholtz rigidity]\label{thm:helmholtz}
Let $S\subset\Z$ be finite, $R=S\cup(-S)$, and let $F:O\subset\Hs\to\Hs$ be a regular $S$-memory field on a convex open set. Assume $(\mathrm{Sep}_R)$. Then the following are equivalent:
\begin{enumerate}[label=(\roman*)]
\item $F$ is potential;
\item $DF(u)$ is self-adjoint for every $u\in O$;
\item for every $u\in O$, $k\in R$, and a.e.\ $\omega$,
\begin{equation}\label{eq:shifted-adjoint}
 A_k[u](\omega)=A_{-k}[u](\theta^k\omega)^*.
\end{equation}
\end{enumerate}
When these conditions hold, a potential is given by \eqref{eq:radial-potential}.
\end{theorem}

\begin{proof}
The equivalence of (i) and (ii) is Proposition~\ref{prop:Vainberg}. By Corollary~\ref{cor:sum-adjoint},
\[
 DF(u)^*=\sum_{k\in R}M_{A_{-k}[u]^*\circ\theta^k}U^k.
\]
Hence
\[
 DF(u)-DF(u)^*
 =\sum_{k\in R}M_{A_k[u]-A_{-k}[u]^*\circ\theta^k}U^k.
\]
Theorem~\ref{thm:identifiability} shows that this operator vanishes exactly when every coefficient vanishes, which is \eqref{eq:shifted-adjoint}.
\end{proof}

\begin{corollary}[One-sided smooth-memory rigidity]\label{cor:one-sided-smooth}
Assume the hypotheses of Theorem~\ref{thm:helmholtz} and $S\cap(-S)\subset\{0\}$. If $F$ is potential, then
\[
 A_k[u]=0\quad\text{a.e.\ for every }u\in O,\ k\in S\setminus\{0\}.
\]
Thus every genuinely active one-sided derivative channel obstructs a global section potential.
\end{corollary}

\begin{example}[Forward implicit residual]\label{ex:forward-residual}
Let $g:H\to H$ be $C^1$ and set
\[
 G(u)=\frac{Uu-u}{\tau}+g(Uu).
\]
Then
\[
 DG(u)=M_{\tau^{-1}I+Dg(Uu)}U-\tau^{-1}I.
\]
Under $(\mathrm{Sep}_{\{-1,0,1\}})$, potentiality forces $Dg(Uu)=-\tau^{-1}I$. If constant sections are admissible on a connected fibre domain, then $g(y)=-\tau^{-1}y+b$ and the forward coupling cancels. Hence a genuinely nonlinear forward implicit residual may be fibrewise variational while failing to be section-potential.
\end{example}

\subsection{A positive two-sided example}
Let
\begin{equation}\label{eq:edge-energy}
 E(u)=\int_\Omega\bigl[V_\omega(u(\omega))+W_\omega(u(\omega),u(\theta\omega))\bigr]d\Pp(\omega)
\end{equation}
under standard differentiability and growth assumptions. Then
\[
 \nabla E(u)(\omega)=\nabla V_\omega(u(\omega))
 +D_1W_\omega(u(\omega),u(\theta\omega))
 +D_2W_{\theta^{-1}\omega}(u(\theta^{-1}\omega),u(\omega)).
\]
Thus a global edge action automatically creates forward and backward channels. Equation \eqref{eq:shifted-adjoint} is the infinitesimal compatibility law behind this pairing.

\subsection{Periodic aliasing}
On a period-$q$ component, the Helmholtz law acts on the residue-class aggregates of Remark~\ref{rem:aliasing}. Thus $(\mathrm{Sep}_R)$ is an identifiability assumption rather than a cosmetic freeness condition.

\section{Global memory-reversal rigidity}

The smooth theorem of Section 5 detects time reversal through the derivative coefficients. We now prove a global finite-difference classification that uses only the input-output dependence of the section operator itself.

\subsection{Finite-memory maps and active lags}
Let $S\subset\Z$ be finite. A continuous map $T:\Hs\to\Hs$ is a \emph{finite $S$-memory map} if it admits a Carath\'eodory representation
\begin{equation}\label{eq:finite-S-rep}
 (Tu)(\omega)=t_\omega\bigl((u(\theta^s\omega))_{s\in S}\bigr),
\end{equation}
where $t_\omega:H^S\to H$ is continuous for a.e.\ $\omega$ and $\omega\mapsto t_\omega(z)$ is measurable for each fixed $z\in H^S$.

For $k\in S$, call the $k$th lag \emph{inactive} if, outside one null set, $t_\omega$ is independent of its $k$th coordinate. Otherwise $k$ is \emph{active}. Equivalently, fixing a countable dense set $D\subset H$, activity means that for some two tuples in $D^S$ differing only in the $k$th coordinate, the corresponding outputs differ on a set of positive measure. We write
\[
 \operatorname{Act}(T)=\{k\in S:k\text{ is active}\}.
\]
For finite $S$, put
\begin{equation}\label{eq:F_S}
 F_S:=\{0\}\cup S\cup(S+S),
 \qquad S+S:=\{s+t:s,t\in S\}.
\end{equation}

\begin{theorem}[Global memory-reversal rigidity]\label{thm:memory-reversal}
Assume $(\mathrm{Sep}_{F_S})$. Let $T:\Hs\to\Hs$ be a finite $S$-memory map. If there exists $\Psi\in C^1(\Hs)$ such that
\[
 T=\nabla\Psi,
\]
then
\begin{equation}\label{eq:active-reversal}
 \boxed{\operatorname{Act}(T)=-\operatorname{Act}(T).}
\end{equation}
In particular, every genuinely active lag $k$ forces the reverse lag $-k$ to be present and genuinely active.
\end{theorem}

\begin{proof}
For $k=0$ the claim is tautological. Fix $k\in S\setminus\{0\}$. We prove the contrapositive implication
\[
 -k\notin\operatorname{Act}(T)\quad\Longrightarrow\quad
 k\notin\operatorname{Act}(T),
\]
interpreting an absent lag $-k\notin S$ as inactive.

By the same finite standard-Borel separation argument used in Theorem~\ref{thm:identifiability}, $(\mathrm{Sep}_{F_S})$ yields a countable family of measurable sets $E$ covering $\Omega$ modulo null sets such that
\begin{equation}\label{eq:FS-separated}
 \{\theta^aE:a\in F_S\}
\end{equation}
are pairwise disjoint. Fix one such $E$ and set
\[
 N=L^2(E;H),\qquad M=L^2(\theta^kE;H).
\]

Suppose that the reverse lag $-k$ is inactive. If an input is perturbed by $h\in N$ and $\eta\in\theta^kE$, then the $s$th memory argument of the output at $\eta$ lies in
\[
 \theta^s(\theta^kE)=\theta^{k+s}E.
\]
When $k+s\ne0$, the exponent $k+s$ belongs to $S+S\subset F_S$, so \eqref{eq:FS-separated} makes $\theta^{k+s}E$ disjoint from $E$. The only possible channel through which a perturbation supported on $E$ could reach output on $\theta^kE$ is therefore $s=-k$; by hypothesis that coordinate is absent or inactive. Hence
\[
 P_MT(u+h)=P_MT(u)
 \qquad\forall u\in\Hs,\quad h\in N,
\]
so $N\notinf_TM$.

The variational reciprocity principle, Proposition~\ref{prop:variational-reciprocity}, gives $M\notinf_TN$. Thus perturbing the input on $\theta^kE$ never changes the output on $E$.

It remains to translate this support-local independence into inactivity of the $k$th pointwise variable. Because the sets $\theta^sE$, $s\in S$, are pairwise disjoint, simple sections can prescribe independently on these sets any finite tuple of values from a fixed countable dense subset $D\subset H$. For two tuples $z,z'\in D^S$ differing only in coordinate $k$, choose two such sections that agree off $\theta^kE$. The relation $M\notinf_TN$ shows, outside a null set independent of the countably many choices,
\[
 t_\omega(z)=t_\omega(z')\qquad\text{for a.e.\ }\omega\in E.
\]
Continuity of $t_\omega$ extends this equality from $D^S$ to all tuples in $H^S$. Hence lag $k$ is inactive on $E$. Since the separating family covers $\Omega$ modulo null sets, $k$ is inactive globally.

We have proved that inactivity of $-k$ implies inactivity of $k$, hence activity of $k$ implies activity of $-k$. Therefore $\operatorname{Act}(T)=-\operatorname{Act}(T)$.
\end{proof}

\begin{remark}[Smooth and global reversal]\label{rem:smooth-global-reversal}
Theorem~\ref{thm:helmholtz} gives the differential law
\[
 A_k[u](\omega)=A_{-k}[u](\theta^k\omega)^*.
\]
Theorem~\ref{thm:memory-reversal} is its global support-level counterpart: it does not identify coefficients, but it forces the set of genuinely active finite-memory variables itself to be closed under time reversal. No derivative of $T$ is used.
\end{remark}

\begin{corollary}[One-sided global variational rigidity]\label{cor:one-sided-variational}
Assume the hypotheses of Theorem~\ref{thm:memory-reversal} and
\[
 S\cap(-S)\subset\{0\}.
\]
Then every nonzero lag is inactive. In particular, for $S=\{0,1,\ldots,m\}$ and $(\mathrm{Sep}_{\{0,1,\ldots,2m\}})$ there is a measurable map $s_\omega:H\to H$ such that
\begin{equation}\label{eq:local-T}
 (Tu)(\omega)=s_\omega(u(\omega))
 \qquad\text{a.e.\ for every }u\in\Hs.
\end{equation}
\end{corollary}

\begin{proof}
If a nonzero $k\in S$ were active, Theorem~\ref{thm:memory-reversal} would make $-k$ active and hence an element of $S$, contradicting $S\cap(-S)\subset\{0\}$. For $S=\{0,\ldots,m\}$, all positive coordinates are therefore inactive. A countable-dense-set argument as in the proof above gives a common null set on which $t_\omega$ depends only on the zeroth coordinate; set $s_\omega(z)=t_\omega(z,0,\ldots,0)$.
\end{proof}

\subsection{Continuous proximity operators as consequences}
We call $T:\K\to\K$ a \emph{continuous proximity operator} if $T$ is continuous and there is a proper, possibly nonconvex, penalty $\phi$ such that
\[
 T(x)\in\arg\min_y\left\{\frac12\|x-y\|^2+\phi(y)\right\}
 \qquad\forall x\in\K.
\]
Gribonval and Nikolova \cite{GribonvalNikolova2020} show that every such $T$ is the gradient of a convex $C^1$ potential.

\begin{corollary}[Finite-memory proximal reversal]\label{cor:prox-memory-reversal}
Under $(\mathrm{Sep}_{F_S})$, every continuous proximity operator on $\Hs$ that admits a finite $S$-memory representation satisfies
\[
 \operatorname{Act}(T)=-\operatorname{Act}(T).
\]
Consequently, if $S\cap(-S)\subset\{0\}$, all nonzero lags are inactive.
\end{corollary}

\begin{corollary}[Shifted fibre proximal maps are not global proximal maps]\label{cor:shifted-fibre-prox}
Let $\phi_\omega\in\Gamma_0(H)$ be a measurable family whose pointwise proximal maps $p_\omega=\prox_{\phi_\omega}$ define
\begin{equation}\label{eq:shifted-fibre-prox}
 (Tu)(\omega)=p_\omega(u(\theta\omega))
\end{equation}
on $\Hs$. Assume $(\mathrm{Sep}_{\{0,1,2\}})$. If $p_\omega$ is nonconstant on a set of positive measure, then there is no proper penalty $\Phi$ for which $T$ is a continuous proximity operator on $\Hs$. In particular, there is no $\Phi\in\Gamma_0(\Hs)$ with $T=\prox_\Phi$.
\end{corollary}

\begin{proof}
The representation \eqref{eq:shifted-fibre-prox} has memory set $S=\{1\}$. If $T$ were a continuous proximity operator, Corollary~\ref{cor:prox-memory-reversal} would imply that lag $1$ is inactive because the reverse lag $-1$ is absent. Thus $p_\omega$ must be constant for almost every $\omega$. More explicitly, if nonconstancy held on a positive-measure set, a countable dense set in $H$ would provide $v,w$ and a positive-measure subset on which $p_\omega(v)\ne p_\omega(w)$, contradicting inactivity of the only memory coordinate.
\end{proof}

\begin{remark}[Scope]\label{rem:prox-scope}
The conclusion concerns representation as one ordinary proximity operator on the same stationary section space. It does not exclude variational principles on enlarged trajectory spaces, nor selfdual or Fitzpatrick-type representations for more general monotone dynamics.
\end{remark}

\subsection{Continuity is sharp for proximal selections}
The continuity assumption in the proximal consequence cannot be removed merely by selecting one minimizer from a multivalued nonconvex proximity correspondence.

\begin{example}[A discontinuous proximal selection with directed dependency]\label{ex:discontinuous-prox}
Fix $\lambda>0$ and define on $\R^2$
\[
 \phi(y_1,y_2)=\lambda\,\mathbf 1_{\{y_1\ne0\}}.
\]
Let $a=\sqrt{2\lambda}$. The proximal minimization separates. Its second coordinate is uniquely $y_2=x_2$. For the first coordinate, the unique minimizer is $0$ when $|x_1|<a$, the unique minimizer is $x_1$ when $|x_1|>a$, and at $|x_1|=a$ both $0$ and $x_1$ are global minimizers.

Choose a single-valued proximal selection $T$ that agrees with these unique minimizers away from the tie set and, at $x_1=a$, uses
\[
 T_1(a,x_2)=
 \begin{cases}
 a,&x_2\ge0,\\
 0,&x_2<0,
 \end{cases}
 \qquad T_2(x_1,x_2)=x_2.
\]
At $x_1=-a$ choose either valid minimizer by any fixed rule. Then $T(x)$ is a global proximal minimizer for every $x$ but is discontinuous.

Let $M=\operatorname{span}(e_1)$ and $N=\operatorname{span}(e_2)$. Since $T_2(x)=x_2$, perturbations in $M$ never change the $N$-output, so $M\notinf_TN$. At $x_1=a$, changing the sign of $x_2$ changes $T_1$, so $N\notinf_TM$ fails. Thus dependency reciprocity can fail for discontinuous proximal selections.
\end{example}

\section{Strong compactness of full measurable-selection lifts}

This section records a complementary obstruction that does not use dynamics.

\begin{lemma}[Closedness and decomposability]\label{lem:selector-closed}
Let $C:\Omega\rightrightarrows X$ be measurable, nonempty closed-valued and $p$-integrably bounded. Then $S_p(C)$ is nonempty, $D_p$-closed and decomposable.
\end{lemma}
\begin{proof}
A Castaing selection belongs to $L^p$ by the integrable envelope, so $S_p(C)$ is nonempty. Decomposability is immediate. If $u_n\to u$ in $D_p$, a subsequence converges pointwise a.e.; closedness of the fibres gives $u(\omega)\in C(\omega)$ a.e.
\end{proof}

\begin{lemma}[Decomposable compactness rigidity]\label{lem:decomp-rigidity}
Let $D\subset L^p(\Omega;X)$ be decomposable. If $D$ is relatively compact in $D_p$, then any two $u,v\in D$ satisfy $u=v$ a.e.\ on $\Omega_{\rm na}$.
\end{lemma}
\begin{proof}
If $u$ and $v$ differ on a positive-measure subset of the nonatomic component, then for some $\delta>0$ there is a nonatomic set $B$ of measure $\beta>0$ on which $d(u,v)\ge\delta$. Choose measurable $E_n\subset B$ with $\Pp(E_n)=\beta/2$ and $\Pp(E_n\triangle E_m)=\beta/2$ for $n\ne m$. Paste $u$ on $E_n$ and $v$ on $E_n^c$ to obtain $w_n\in D$. Then
\[
 D_p(w_n,w_m)^p\ge\delta^p\beta/2,
\]
contradicting relative compactness.
\end{proof}

\begin{theorem}[Strong compactness rigidity]\label{thm:selector-compactness}
Let $C:\Omega\rightrightarrows X$ be measurable, nonempty compact-valued and $p$-integrably bounded. Then the following are equivalent:
\begin{enumerate}[label=(\roman*)]
\item $S_p(C)$ is compact in $L^p(\Omega;X)$;
\item $C(\omega)$ is a singleton for a.e.\ $\omega\in\Omega_{\rm na}$.
\end{enumerate}
In particular, on a purely nonatomic probability space, the full selection lift is strongly compact if and only if the correspondence is essentially single-valued.
\end{theorem}

\begin{proof}
If $S_p(C)$ is compact, Lemma~\ref{lem:decomp-rigidity} implies that all Castaing selections agree a.e.\ on $\Omega_{\rm na}$, hence the fibres there are singletons.

Conversely, assume $C(\omega)=\{c(\omega)\}$ on the nonatomic component. On each atom $A_j$, Lemma~\ref{lem:atom-constant} identifies the fibre correspondence a.e.\ with a compact set $C_j\subset X$. Every selection is represented by $c$ on $\Omega_{\rm na}$ and a coordinate $x_j\in C_j$ on each atom. If $p_j=\Pp(A_j)$ and $r_j=\sup_{x\in C_j}d(x,x_\circ)$, integrable boundedness gives
\[
 \sum_jp_jr_j^p<\infty.
\]
For any sequence of selections, compactness of each $C_j$ and a diagonal argument give coordinatewise convergence along a subsequence. The uniform tail estimate
\[
 \sum_{j>N}p_jd(x_{n,j},x_j)^p
 \le 2^p\sum_{j>N}p_jr_j^p\to0
\]
upgrades this to strong $L^p$ convergence.
\end{proof}

\begin{remark}\label{rem:selector-positioning}
The measurable-selection and decomposability ingredients are classical \cite{CastaingValadier1977}. The theorem is used here as the strong-compactness counterpart of the variational lifting obstructions, not as a claim that decomposable-set theory begins with the present paper.
\end{remark}

\section{Branch pasting and strong ergodicity}

We now identify the dynamical compactness threshold for a two-branch class of approximate solutions.

\subsection{One transformation}
Let $(X,d)$ be Polish, let $Y$ be Banach, and consider a one-step residual
\begin{equation}\label{eq:G-one-step}
 (G\Gamma)(\omega)=G_\omega(\Gamma(\omega),\Gamma(\theta\omega)).
\end{equation}
Assume two coherent branches $\xi_0,\xi_1\in L^p(\Omega;X)$ satisfy
\begin{equation}\label{eq:coherent-branches}
 G_\omega(\xi_i(\omega),\xi_i(\theta\omega))=0
 \quad\text{a.e., }i=0,1,
\end{equation}
and
\begin{equation}\label{eq:pointwise-separation}
 d(\xi_0(\omega),\xi_1(\omega))>0\quad\text{a.e.}
\end{equation}
Define
\[
 q(\omega)=\max_{i\ne j}\|G_\omega(\xi_i(\omega),\xi_j(\theta\omega))\|_Y,
\]
\[
 r(\omega)=\min_{i\ne j}\|G_\omega(\xi_i(\omega),\xi_j(\theta\omega))\|_Y.
\]
Assume
\begin{equation}\label{eq:cross-assumptions}
 q\in L^p(\Omega),\qquad r>0\ \text{a.e.}
\end{equation}
For measurable $A\subset\Omega$, let $\Gamma_A=\xi_1$ on $A$ and $\Gamma_A=\xi_0$ on $A^c$, and set $\partial_\theta A=A\triangle\theta^{-1}A$.

\begin{lemma}[Two-sided branch-pasting estimate]\label{lem:branch-two-sided}
Under \eqref{eq:coherent-branches}--\eqref{eq:cross-assumptions},
\begin{equation}\label{eq:branch-two-sided}
 \int_{\partial_\theta A}r^p\,d\Pp
 \le\|G(\Gamma_A)\|_{L^p}^p
 \le\int_{\partial_\theta A}q^p\,d\Pp.
\end{equation}
\end{lemma}
\begin{proof}
Outside $\partial_\theta A$, the labels at $\omega$ and $\theta\omega$ agree, so the residual vanishes by coherence. On the boundary the residual is exactly one of the two cross-branch terms, which lies between $r$ and $q$.
\end{proof}

\begin{theorem}[Residual vanishing equals asymptotic label invariance]\label{thm:residual-ai}
For measurable sets $A_n$,
\begin{equation}\label{eq:residual-ai-iff}
 \|G(\Gamma_{A_n})\|_{L^p}\to0
 \quad\Longleftrightarrow\quad
 \Pp(A_n\triangle\theta^{-1}A_n)\to0.
\end{equation}
\end{theorem}
\begin{proof}
If the boundary measure tends to zero, the upper bound in \eqref{eq:branch-two-sided} and absolute continuity of the integral of $q^p$ give residual convergence to zero.

Conversely, suppose the residual tends to zero. The lower bound gives
\[
 \int_{\partial_\theta A_n}r^p\,d\Pp\to0.
\]
Given $\varepsilon>0$, because $r>0$ a.e.\ choose $\delta>0$ with $\Pp(r<\delta)<\varepsilon/2$. Then
\[
 \Pp(\partial_\theta A_n)
 \le\Pp(r<\delta)+\delta^{-p}\int_{\partial_\theta A_n}r^p\,d\Pp<\varepsilon
\]
for all sufficiently large $n$.
\end{proof}

\subsection{Finitely generated group actions}
Let a finitely generated group $\Gamma$ act by p.m.p.\ transformations on $(\Omega,\Ff,\Pp)$, and let $S=S^{-1}$ be a finite generating set. For each $s\in S$, let $Y_s$ be Banach and let
\[
 G^s_\omega:X\times X\to Y_s
\]
be a measurable one-step residual. Assume the same two branches $\xi_0,\xi_1$ are coherent for every generator:
\[
 G^s_\omega(\xi_i(\omega),\xi_i(s\omega))=0
 \quad\text{a.e., }i=0,1.
\]
Define $q_s,r_s$ from the cross-branch terms as above and assume $q_s\in L^p$ and $r_s>0$ a.e.\ for every $s\in S$. Put
\[
 \|\mathbf G(\Gamma)\|_{p,S}^p
 =\sum_{s\in S}\|G^s(\Gamma)\|_{L^p}^p.
\]

\begin{proposition}[Generator-wise residual criterion]\label{prop:group-residual-ai}
For measurable $A_n$,
\begin{equation}\label{eq:group-ai}
 \|\mathbf G(\Gamma_{A_n})\|_{p,S}\to0
 \quad\Longleftrightarrow\quad
 \Pp(A_n\triangle s^{-1}A_n)\to0
 \quad\forall s\in S.
\end{equation}
The right-hand side is equivalent to asymptotic invariance under every $g\in\Gamma$.
\end{proposition}
\begin{proof}
Apply Theorem~\ref{thm:residual-ai} to each generator. Since $S$ is finite, the generator-wise statements are equivalent to convergence of the total residual. Write $g^{-1}=t_1\cdots t_k$ with $t_j\in S$. The symmetric-difference triangle inequality and measure preservation give
\[
 \begin{aligned}
 \Pp(A_n\triangle g^{-1}A_n)
 &\le \sum_{j=1}^k
 \Pp(t_1\cdots t_{j-1}A_n\triangle t_1\cdots t_jA_n)\\
 &=\sum_{j=1}^k\Pp(A_n\triangle t_jA_n),
 \end{aligned}
\]
which tends to zero because $S=S^{-1}$. Thus generator-wise asymptotic invariance implies invariance for every fixed group element.
\end{proof}

Recall that an ergodic p.m.p.\ action is \emph{strongly ergodic} if every asymptotically invariant sequence $A_n$ is trivial:
\begin{equation}\label{eq:strong-ergodic-def}
 \Pp(A_n)(1-\Pp(A_n))\to0.
\end{equation}
This is classical; see Schmidt \cite{Schmidt1981}.

\begin{corollary}[Strong ergodicity as the branch-pasting compactness threshold]\label{cor:strong-threshold}
Assume the finitely generated p.m.p.\ action above is ergodic and the two-branch hypotheses hold. Then the following are equivalent:
\begin{enumerate}[label=(\roman*)]
\item the action is strongly ergodic;
\item every branch-pasting sequence $\Gamma_{A_n}$ satisfying
\[
 \|\mathbf G(\Gamma_{A_n})\|_{p,S}\to0
\]
is relatively compact in $L^p(\Omega;X)$.
\end{enumerate}
If the action is not strongly ergodic, there is an approximate-zero branch-pasting sequence with no strongly convergent subsequence.
\end{corollary}

\begin{proof}
Assume strong ergodicity and let the residual tend to zero. Proposition~\ref{prop:group-residual-ai} makes $A_n$ asymptotically invariant, so \eqref{eq:strong-ergodic-def} holds. Consider an arbitrary subsequence. It has a further subsequence along which either $\Pp(A_n)\to0$ or $\Pp(A_n)\to1$. Since
\[
 d(\Gamma_{A_n},\xi_0)^p=1_{A_n}d(\xi_1,\xi_0)^p,
\]
and $d(\xi_0,\xi_1)^p$ is integrable, absolute continuity of the integral gives $\Gamma_{A_n}\to\xi_0$ strongly in the first case; similarly $\Gamma_{A_n}\to\xi_1$ in the second. Thus every subsequence has a convergent further subsequence, proving relative compactness.

Conversely, suppose the action is not strongly ergodic. There is an asymptotically invariant sequence $A_n$ which is nontrivial. Passing to a subsequence, for some $\varepsilon>0$,
\begin{equation}\label{eq:nontrivial-mass}
 \Pp(A_n)(1-\Pp(A_n))\ge\varepsilon
\end{equation}
for every $n$. Proposition~\ref{prop:group-residual-ai} gives $\|\mathbf G(\Gamma_{A_n})\|_{p,S}\to0$.

Suppose a subsequence of $\Gamma_{A_n}$ converged strongly. It would be Cauchy. Write
\[
 \rho(\omega)=d(\xi_0(\omega),\xi_1(\omega))>0\quad\text{a.e.}
\]
Then
\begin{equation}\label{eq:weighted-label-distance}
 D_p(\Gamma_A,\Gamma_B)^p
 =\int_{A\triangle B}\rho^p\,d\Pp.
\end{equation}
A Cauchy sequence for this weighted metric is Cauchy in measure for the indicators: given $\eta>0$, choose $\delta>0$ with $\Pp(\rho<\delta)<\eta/2$, and use
\[
 \Pp(A\triangle B)
 \le\Pp(\rho<\delta)+\delta^{-p}\int_{A\triangle B}\rho^p\,d\Pp.
\]
Hence the indicators have a subsequence converging in measure, and therefore in $L^1$, to $1_A$ for some measurable $A$. Asymptotic invariance passes to the limit, so $A$ is invariant. Ergodicity gives $\Pp(A)\in\{0,1\}$. But $\Pp(A_n)\to\Pp(A)$ along the convergent subsequence, contradicting \eqref{eq:nontrivial-mass}. Thus no strongly convergent subsequence exists.
\end{proof}

\begin{corollary}[Aperiodic ergodic $\Z$-actions]\label{cor:Z-branch}
Let $\theta$ be an aperiodic ergodic p.m.p.\ transformation on a standard nonatomic probability space. Under the two-branch hypotheses, there are branch-pasting sections with residual tending to zero and no strongly convergent subsequence.
\end{corollary}
\begin{proof}
Rokhlin towers produce nontrivial asymptotically invariant label sets; equivalently, such an aperiodic ergodic $\Z$-action is not strongly ergodic. Apply Corollary~\ref{cor:strong-threshold}.
\end{proof}

\subsection{Properness and resonant Fredholmness}
\begin{corollary}[Nonproperness]\label{cor:nonproper}
In the setting of Corollary~\ref{cor:Z-branch}, if $G$ is a continuous single-valued map on a metric section space containing the bad sequence, then $G$ is not proper.
\end{corollary}
\begin{proof}
If $G(\Gamma_n)\to0$, the compact set $\{0\}\cup\{G(\Gamma_n):n\ge1\}$ has inverse image containing a nonprecompact sequence.
\end{proof}

\begin{proposition}\label{prop:UminusI}
For an aperiodic ergodic p.m.p.\ transformation, $\Ran(U-I)$ is not closed on $L^2(\Omega)$.
\end{proposition}
\begin{proof}
Use nontrivial almost-invariant sets $A_n$ and center their indicators:
\[
 f_n=\frac{1_{A_n}-\Pp(A_n)}{\sqrt{\Pp(A_n)(1-\Pp(A_n))}}.
\]
After passing to a nontrivial subsequence, $\|f_n\|_2=1$, $f_n\in L^2_0$, and $\|(U-I)f_n\|_2\to0$. Ergodicity makes $U-I$ injective on $L^2_0$. Closed range would imply a lower bound there, a contradiction.
\end{proof}

\begin{corollary}[Resonant non-Fredholmness]\label{cor:nonfredholm}
If a linearization has the form $DG(\Gamma_*)=B(U-I)$ with $B$ boundedly invertible, then it has nonclosed range and is not Fredholm. The same conclusion holds for compact perturbations.
\end{corollary}

\section{Variational and Wasserstein examples}

The examples illustrate the structural theorems; they are not additional existence or bifurcation theories.

\subsection{Implicit Hilbert-space steps}
Let $V_\omega:H\to\R$ be $C^2$ and consider
\[
 y\in\arg\min_z\left\{\frac{1}{2\tau}\|z-x\|^2+V_\omega(z)\right\}.
\]
At a smooth critical point,
\[
 \frac{y-x}{\tau}+\nabla V_\omega(y)=0.
\]
The corresponding one-step section residual is
\[
 (Gu)(\omega)=\frac{u(\theta\omega)-u(\omega)}{\tau}+\nabla V_\omega(u(\theta\omega)).
\]
Its derivative is
\[
 DG(u)=M_{\tau^{-1}I+\nabla^2V_\omega(u(\theta\omega))}U-\tau^{-1}I.
\]
Under finite separation, Corollary~\ref{cor:one-sided-smooth} forces the forward coefficient to vanish if a global section potential exists. Thus genuine nonlinear one-step fibrewise variational structure need not lift to a stationary section potential.

\subsection{Compact multivalued variational resolvents}
Let
\[
 \Phi_\omega(x)=\arg\min_y A_\omega(y;x)
\]
be measurable and compact-valued. For a fixed input section $u$ define $C_u(\omega)=\Phi_\omega(u(\omega))$. If $C_u$ is $p$-integrably bounded, Theorem~\ref{thm:selector-compactness} gives
\[
 S_p(C_u)\ \text{compact}
 \quad\Longleftrightarrow\quad
 C_u(\omega)\ \text{singleton a.e.\ on }\Omega_{\rm na}.
\]
Thus compactness of every fibre response set does not protect the full measurable response family from strong noncompactness.

\subsection{A nonconvex Wasserstein minimizing step}
Let $\cP_2(\R)$ be the quadratic Wasserstein space. Fix $\tau>0$, $\beta\tau>1$ and $\kappa>0$, and define
\[
 \mathcal E(\nu)=
 \begin{cases}
 \dfrac\beta4\displaystyle\int(x^2-1)^2\,d\nu(x)+\kappa\Var(\nu),&\displaystyle\int x^4\,d\nu<\infty,\\[1ex]
 +\infty,&\text{otherwise}.
 \end{cases}
\]
Consider one minimizing-movement step from $\delta_0$:
\[
 J_\tau(\delta_0)=\arg\min_{\nu\in\cP_2(\R)}
 \left\{\frac1{2\tau}W_2^2(\nu,\delta_0)+\mathcal E(\nu)\right\}.
\]
Since $W_2^2(\nu,\delta_0)=\int x^2\,d\nu$, define
\[
 \varphi_\tau(x)=\frac{x^2}{2\tau}+\frac\beta4(x^2-1)^2.
\]
The global minima are $\pm a$, where
\[
 a=\sqrt{1-\frac1{\beta\tau}}.
\]
The integral of $\varphi_\tau$ is minimized exactly by measures supported on $\{-a,a\}$, and the positive variance penalty eliminates nontrivial mixtures. Therefore
\[
 J_\tau(\delta_0)=\{\delta_{-a},\delta_a\}.
\]
Putting this same fibre problem over a nonatomic probability space gives a two-point compact fibre correspondence whose full $L^p(W_2)$ selection lift is strongly noncompact by Theorem~\ref{thm:selector-compactness}. This is only an illustration of the lifting mechanism.

\section{Scope, sharpness and limitations}

\subsection{Reciprocity versus classical gradient symmetry}
Proposition~\ref{prop:variational-reciprocity} is deliberately formulated at the level of finite differences and closed-subspace influence. In $C^2$ settings it is compatible with the classical symmetry of the Hessian. Its role here is to organize graph decomposition and dynamical no-go consequences, not to claim a new abstract notion of exact differential form.

\subsection{Continuity and set-valued proximal correspondences}
Continuity is essential for the passage from a single-valued proximity operator to a convex $C^1$ potential. Example~\ref{ex:discontinuous-prox} shows that a discontinuous selection from a multivalued nonconvex proximity correspondence can carry directed dependency. A systematic set-valued reciprocity theory is outside the scope of this paper.

\subsection{Finite memory}
The dynamical identifiability theorem is finite-channel. Infinite memory would require summability and a different uniqueness mechanism; recurrence prevents the separating-set argument from extending verbatim to infinitely many iterates.

\subsection{Invertibility}
The shifted-adjoint law uses $(U^k)^*=U^{-k}$. For noninvertible dynamics, the adjoint is a transfer operator rather than a reverse Koopman power. Hence the adjoint of a forward channel generally leaves the present channel family. The noninvertible problem therefore requires a different mixed Koopman--transfer representation theory and is not treated here.

\subsection{Strong ergodicity}
Corollary~\ref{cor:strong-threshold} concerns the branch-pasting class generated by two coherent branches. It does not characterize arbitrary approximate-zero sequences. Strong ergodicity is the exact threshold only within this structural class under pointwise cross-residual nondegeneracy.

\subsection{Full selection lifts}
Theorem~\ref{thm:selector-compactness} concerns the complete decomposable family of measurable selectors. Smaller families selected by regularity, monotonicity or dynamical compatibility may remain compact in the presence of fibre branching.

\subsection{Finite-p regime}
The atomic compactness argument in Theorem~\ref{thm:selector-compactness} uses finite-$p$ tail summability. The $L^\infty$ case requires a different criterion.

\subsection{Wasserstein scope}
The Wasserstein example applies the compact-selection theorem to one explicit minimizing correspondence. We do not construct a Wasserstein degree, a general random minimizing-movement theory or a new Wasserstein bifurcation invariant.

\subsection{Overall interpretation}
The paper isolates two complementary structural costs of lifting. On the variational side, global potential structure imposes time-reversal closure on finite dynamical memory: smoothly through the shifted-adjoint channel law and globally through $\operatorname{Act}(T)=-\operatorname{Act}(T)$. One-sided gradient and proximal memories therefore collapse to present-state locality. On the compactness side, measurable branching and asymptotically invariant labels generate noncompact section families. These mechanisms explain why favourable fibrewise properties cannot in general be transferred unchanged to dynamically coupled spaces of measurable sections.

\appendix

\section{Differentiability of finite-memory Nemytskii fields}
The Helmholtz theorem is intentionally formulated through the derivative representation \eqref{eq:regular-derivative}, rather than through a maximal superposition-operator theorem. A pointwise field of the form \eqref{eq:finite-memory} falls under Definition~\ref{def:regular-memory} whenever the associated Nemytskii map is $C^1$ on the section domain and its partial derivatives induce strong-operator measurable, essentially bounded multiplication fields. Standard Nemytskii differentiability results provide many sufficient conditions. No such maximal theorem is used in the proof of Theorem~\ref{thm:helmholtz}; only the displayed derivative representation is required.

\section{Selection and atomic technicalities}
We use two classical facts. First, measurable nonempty closed-valued multifunctions into Polish spaces admit Castaing representations \cite{CastaingValadier1977}. Second, the decomposition \eqref{eq:atomic-decomp} is unique modulo null sets up to reordering of atoms. On an atom, every measurable selection is a.e.\ constant by Lemma~\ref{lem:atom-constant}. If $C$ is compact-valued, the induced atomic fibre $C_j$ is compact, and the integrable envelope gives the tail estimate used in Theorem~\ref{thm:selector-compactness}.

\section{Separating sets and periodic aliasing}
The proof of Theorem~\ref{thm:identifiability} extends to any finite family $T_1,\dots,T_m$ of invertible measurable p.m.p.\ transformations satisfying
\[
 \Pp\{\omega:T_i\omega=T_j\omega\}=0\qquad(i\ne j).
\]
Then $\sum_iM_{A_i}U_{T_i}=0$ implies $A_i=0$ a.e.\ for every $i$. For a period-$q$ component the natural coefficients are residue-class aggregates, and self-adjointness pairs residue $r$ with $-r$ after the appropriate shift.

\section{Rokhlin construction for integer actions}
Let $\theta$ be aperiodic p.m.p.\ and fix $\alpha\in(0,1)$. Rokhlin towers with heights $N\to\infty$ give bases $B_N$ such that
\[
 B_N,\theta B_N,\ldots,\theta^{N-1}B_N
\]
are pairwise disjoint and their union has measure tending to one. With $m_N=\lfloor\alpha N\rfloor$ and
\[
 A_N=\bigcup_{j=0}^{m_N-1}\theta^jB_N,
\]
one has $\Pp(A_N)\to\alpha$ and
\[
 \Pp(A_N\triangle\theta A_N)\le 2\Pp(B_N)\le 2/N\to0.
\]
This is the explicit almost-invariant sequence behind Corollary~\ref{cor:Z-branch}.

\section{Wasserstein double-well calculation}
For
\[
 \varphi_\tau(x)=\frac{x^2}{2\tau}+\frac\beta4(x^2-1)^2,
\]
with $\beta\tau>1$,
\[
 \varphi_\tau'(x)=x\bigl(\tau^{-1}+\beta(x^2-1)\bigr).
\]
The critical points are $0$ and $\pm a$ with $a^2=1-(\beta\tau)^{-1}$. Moreover,
\[
 \varphi_\tau''(0)=\tau^{-1}-\beta<0,
 \qquad
 \varphi_\tau''(\pm a)=2\beta a^2>0.
\]
Since $\varphi_\tau(x)\to\infty$ as $|x|\to\infty$, $\pm a$ are the global minima. Equality in
\[
 \int\varphi_\tau\,d\nu\ge\varphi_\tau(a)
\]
requires $\operatorname{supp}\nu\subset\{-a,a\}$. Every nontrivial mixture has positive variance, whereas the two Dirac masses have variance zero. Thus the minimizing correspondence is exactly $\{\delta_{-a},\delta_a\}$.

\section{Assumption and dependency table}
\begin{center}
\small
\begin{tabular}{p{0.25\textwidth}p{0.65\textwidth}}
\toprule
Result & Essential assumptions / sharpness status\\
\midrule
Proposition~\ref{prop:variational-reciprocity} & Real Hilbert space; global $C^1$ potential. Convexity and orthogonality of the tested subspaces are unnecessary.\\
Theorem~\ref{thm:graph-classification} & Finite orthogonal block decomposition; global $C^1$ potential.\\
Theorem~\ref{thm:prox-components} & $\Phi\in\Gamma_0$; orthogonal finite component decomposition.\\
Theorem~\ref{thm:identifiability} & Finite memory family and $(\mathrm{Sep}_F)$; strong-operator measurable essentially bounded coefficients. The separation assumption is necessary and sufficient.\\
Theorem~\ref{thm:helmholtz} & Hilbert $L^2$ section space; regular finite-memory $C^1$ field; convex domain; finite-channel separation.\\
Theorem~\ref{thm:memory-reversal} & Global $C^1$ potential; continuous Carath\'eodory finite $S$-memory representation; $(\mathrm{Sep}_{F_S})$ with $F_S=\{0\}\cup S\cup(S+S)$.\\
Corollary~\ref{cor:prox-memory-reversal} & Continuous single-valued proximity operator plus finite-memory representation. Example~\ref{ex:discontinuous-prox} shows discontinuous selections can violate reciprocity.\\
Theorem~\ref{thm:selector-compactness} & Standard probability base; Polish fibres; compact measurable values; $p$-integrable envelope; $1\le p<\infty$.\\
Theorem~\ref{thm:residual-ai} & Two coherent branches; pointwise positive cross-residual lower envelope; $L^p$ upper envelope. No uniform positive lower constant is needed.\\
Corollary~\ref{cor:strong-threshold} & Finitely generated ergodic p.m.p.\ action; the preceding two-branch nondegeneracy conditions.\\
\bottomrule
\end{tabular}
\end{center}

\section*{Funding}
This work was supported in part by the National Natural Science Foundation of China [grant numbers 62276135].

\section*{Declaration of generative AI and AI-assisted technologies in the manuscript preparation process}
During the preparation of this work, the authors used ChatGPT (OpenAI) to assist with manuscript organization, language refinement, and presentation. The authors carefully reviewed and verified the mathematical arguments, results, and references, revised the content as necessary, and take full responsibility for the final manuscript.

\end{document}